\documentclass[11pt]{amsart}
 
\theoremstyle{plain}
\newtheorem{thm}{Theorem}[section]
\newtheorem{theorem}[thm]{Theorem}

\newtheorem{lemma}[thm]{Lemma}

\newtheorem{proposition}[thm]{Proposition}
\theoremstyle{definition}
\newtheorem{remark}[thm]{Remark}

\newtheorem{question}[thm]{Question}

\newtheorem{problem}[thm]{Problem}

\numberwithin{equation}{section}
\title [primitive automorphisms of positive entropy]{Explicit examples of rational and Calabi-Yau threefolds with primitive automorphisms 
of positive entropy}
\author{Keiji Oguiso} 
 \author{Tuyen Trung Truong}
\address{Department of Mathematics, Osaka University, Toyonaka 560-0043, Osaka, Japan and Korea Institute for Advanced Study, Hoegiro 87, Seoul, 133-722, Korea}
 \email{oguiso@math.sci.osaka-u.ac.jp}   
       \address{Department of Mathematics, Syracuse University, Syracuse, NY 13244, USA}
 \email{tutruong@syr.edu}
 \dedicatory{Dedicated to Professor Fr\'ed\'eric Campana on the occasion of his sixtieth birthday.}

\thanks{The first author is supported by JSPS Grant-in-Aid (S) No 25220701,
JSPS Grant-in-Aid (S) No 22224001, JSPS Grant-in-Aid (B) No 22340009, and by KIAS Scholar Program.}

    \date{\today}

\begin{document}

\maketitle

\begin{abstract} We present the first explicit examples of a rational threefold and a Calabi-Yau threefold, admitting biregular automorphisms of positive entropy not preserving any dominant rational maps to  lower positive dimensional varieties. These examples are also the first whose dynamical degrees are not Salem numbers. Crucial parts are the rationality of the quotient threefold of a certain $3$-dimensional torus of product type and a numerical criterion of primitivity of birational automorphisms in terms of dynamical degrees. 
\end{abstract}

\section{Introduction}

Unless stated otherwise, we work in the category of projective varieties defined over the complex number field $\mathbf C$. Our main result is 
Theorem \ref{main} below. We first introduce some background for the study in this paper. 

\subsection{General Problem.} Complex dynamics of biholomorphic automorphisms of compact K\"ahler surfaces is now fairly well-understood since Cantat \cite{Ca99}. Especially, Bedford-Kim (\cite{bedford-kim1}, \cite{bedford-kim2}, \cite{bedford-kim3}) and McMullen (\cite{mcmullen}, \cite{mcmullen1}, \cite{mcmullen2}, \cite{mcmullen3}) show very beautiful aspects of automorphisms of rational surfaces and K3 surfaces. A general theory for the dynamics of biholomorphic automorphisms of higher dimensions were developed by Dinh and Sibony \cite{dinh-sibony3} \cite{dinh-sibony4}, however interesting examples are not yet found enough. Indeed, the following basic problem is completely open:

\begin{problem}\label{Explicit}
Find (many) examples of rational manifolds and Calabi-Yau manifolds admitting {\it primitive} biregular automorphisms of {\it positive entropy}. 
\end{problem}

\begin{remark}\label{PrimitiveEntropy}

(1) Here a dominant selfmap $f$ of a manifold $M$ is {\it imprimitive} if $f$ comes from lower dimensional manifolds, or more precisely, there are a dominant rational map $\varphi : M \cdots \to B$ with $0 <\dim\, B < \dim\, M$ and a dominant rational map $g : B \cdots \to B$ such that $\varphi \circ f = g \circ \varphi$. A dominant selfmap that is not imprimitive is {\it primitive}. This notion was introduced by Zhang \cite{Zh09}. In \cite{Zh09}, it is also proved that if a projective threefold $M$ admits a primitive birational automorphism of infinite order, then $M$ is birationally equivalent to either a $3$-dimensional complex torus, a weak Calabi-Yau threefold or a rationally connected threefold. Here a minimal threefold in the sense of the minimal model theory is called a weak Calabi-Yau threefold if the canonical divisor is numerically trivial and the irregularity is $0$. Needless to say, rational threefolds (resp. smooth Calabi-Yau threefolds) are the most basic examples of rationally connected threefolds (resp. weak Calabi-Yau threefolds). 

(2) {\it Entropy} is an important invariant that measures how fast two general points spread out under the action of $\langle f \rangle$. The original definition is a completely topological one, involving a metric $d(x,y)$ on the manifold and the induced metrics $d_n(x,y)=\max _{0\leq j\leq n}d(f^jx,f^jy)$ (See e.g. Bowen \cite{bowen} for the case of continuous maps, and for the case of rational maps see Friedland \cite{friedland} and Guedj \cite{guedj2}.)  When $f$ is a regular endomorphism of a compact K\"ahler manifold, a fundamental theorem of Gromov \cite{gromov} and Yomdin \cite{yomdin}  says that the entropy of $f$ is 
$$h_{top}(f)={\rm log}\, 
{\max}_{0 \le k \le {\rm dim}\, M} \lambda_k(f)\,\, ,$$
where $\lambda_k(f)$ is the $k$-th {\it dynamical degree}, i.e., the spectral radius of $f^* \vert H^{2k}(M, {\mathbf Z})$. In particular, if $f$ is of positive entropy, then $f$ is of infinite order. The theorem of Gromov and Yomdin was partially extended to the case of meromorphic maps in Dinh-Sibony \cite{dinh-sibony10} \cite{dinh-sibony1} (See also Section 3).

(3) If an automorphism of a compact K\"ahler manifold is imprimitive, then there are also some strong constraints on its dynamical degrees (Theorem \ref{Restriction}). This is an important point of our construction of primitive automorphisms. For example, in dimension $3$, its first and second dynamical degrees must be the same. In our construction in Theorem \ref{main} (see also Proposition \ref{Fibration}), the automorphisms have {\it different first and second dynamical degrees}, therefore it follows that they are primitive. Our theorem \ref{Restriction} is a simple but useful consequence of recent progress in complex dynamics (\cite{dinh-nguyen} and \cite{dinh-nguyen-truong}), which we shall review in Section 3. We also note that in dimension $3$, if the first dynamical degrees are either $1$ or a Salem number, then the first and second dynamical degrees of these examples must be the same (\cite{oguiso-truong}).
\end{remark}

The main aim of this note is to present first explicit examples of a rational threefold and a Calabi-Yau threefold in Problem \ref{Explicit}. 

\subsection{Constructing rational threefolds with rich automorphisms.} One of the essential parts of Theorem \ref{main} is the rationality of the obtained threefold with rich automorphisms. To explain this, we recall two basic constructions of manifolds from a given manifold $V$:

(I) Take a finite successive blow up $M$ of $V$ along smooth centers.

(II) Take a nice resolution $M$ of singularities of the quotient variety $V/G$ by a finite subgroup $G \subset {\rm Aut}\, (V)$. 

In construction (I), it is clear that $M$ is rational if $V = {\mathbf P}^3$ or ${\mathbf P}^2\times {\mathbf P}^1$ or ${\mathbf P}^1\times {\mathbf P}^1\times {\mathbf P}^1$, and in this way, there are constructed some explicit, interesting examples of {\it birational} automorphisms of rational threefolds of infinite orders (Bedford-Kim \cite{bedford-kim}\cite{bedford-kim4}, Perroni-Zhang \cite{perroni-zhang}, Blanc \cite{blanc}). However, they are either only pseudo-automorphisms (i.e., birational selfmaps isomorphic in codimension one) but not biregular, or imprimitive, or of null-entropy. Moreover, in all of these examples the first dynamical degrees are either $1$ or a Salem number, and hence the first and second dynamical degrees of these examples must be the same (compare with Remark \ref{PrimitiveEntropy} (3)).  

In fact, in dimension $\ge 3$, it is rather hard to construct {\it biregular} automorphisms following construction (I). For instance, the following very simple question by Professor Eric Bedford (at a conference in Paris in 2011), which was studied by the second author in \cite{truong}, still remains open: 
\begin{question}\label{truong}
Is there a smooth rational threefold $W$ obtained by a successive blow-up of ${\mathbf P}^3$ along smooth centers such that $W$ admits biregular automorphisms of positive entropy? 
\end{question}
The results in \cite{truong} suggest that the answer to Question \ref{truong} 
could be negative. So, at the moment, construction (I) may not be so promising to Problem \ref{Explicit}. 

In the second construction (II), $M$ has many biregular automorphisms if $V$ has many biregular automorphisms normalizing $G$. As we shall do in one specific case, we may have more chances to find primitive automorphisms of positive entropy among them, by using the method explained in Remark \ref{PrimitiveEntropy} (3). However, the rationality of $V/G$ is highly non-trivial in this construction (even if $V$ itself is rational!). There are now many 
general methods to see if $M$ is uniruled, rationally connected or not (\cite{MM86}, \cite{KMM92}, \cite{KL09}), and some useful methods to conclude $M$ is {\it not} rational (\cite{IM71}, \cite{CG72}, \cite{AM72}, \cite{Ko95}, \cite{Be12}). However, in dimension $\ge 3$, there is essentially no general method to conclude $M$ is rational or unirational, and this is in general very hard (\cite{Ko02}, see also the excellent survey \cite{kollar}). 

\subsection{Main Result.} 

Let 
$$\omega := \frac{-1 + \sqrt{-3}}{2}$$
and 
$$E_{\omega} := {\mathbf C}/({\mathbf Z} + {\mathbf Z}\omega)$$ 
be the elliptic curve of period $\omega$. The elliptic curve $E_{\omega}$ is a very special one that is characterized as the unique elliptic curve admitting an automorphism $c$ of order $3$ such that $c^*\sigma = \omega\sigma$. Here $\sigma$ is a non-zero holomorphic $1$-form on the elliptic curve.  

Let $X$ be the canonical resolution of the quotient threefold 
$$E_{\omega} \times E_{\omega} \times E_{\omega} / \langle {\rm diag}\, (\omega, \omega, \omega) \rangle\,\, ,$$
i.e., the blow-up at the singular points of type $1/3(1,1,1)$
and $Y$ be the canonical resolution of the quotient threefold 
$$E_{\omega} \times E_{\omega} \times E_{\omega} / \langle {\rm diag}\, (-\omega, -\omega, -\omega) \rangle\,\, ,$$
i.e., the blow-up at the singular points of type $1/2(1,1,1)$, $1/3(1,1,1)$, $1/6(1,1,1)$. 

$Y$ is then the canonical resolution of $X/\iota$, where $\iota$ is the involution naturally induced by ${\rm diag}\, (-1, -1, -1)$. 
It is well-known that $X$ is a Calabi-Yau threefold with very special properties, e.g., it is rigid  and plays a crucial role in the classification of fiber space structures on Calabi-Yau threefolds (\cite{Be82}, \cite{OS01} and references therein). Our theorem shows that $X$ has also a rich structure in the complex dynamical view. We remark that also for Calabi-Yau manifolds, it seems much harder to construct biregular automorphisms compared with birational selfmaps (See for instance \cite{Og12}, \cite{CO12}). On the other hand, our new manifold $Y$ seems so far caught no attention. 

Now we state the main result:
\begin{theorem}\label{main}

(1) $Y$ is rational, i.e., birationally equivalent to ${\mathbf P}^3$. 

(2) Both $X$ and $Y$ admit primitive biregular automorphisms of positive entropy whose first and second dynamical degrees are not the same (hence they are not Salem numbers).
\end{theorem}

Our $X$ and $Y$ are the first explicit examples of Calabi-Yau threefolds and rational threefolds which admit biregular, primitive automorphisms of positive entropy. The condition that the first and second dynamical degrees are not the same makes the results in \cite{dinh-sibony3}\cite{dinh-sibony4} and Dinh-de Thelin \cite{dinh-deThelin} applicable to these examples, therefore showing that they have good  dynamical properties. As we explained in the last subsection, the crucial point of the proof is the rationality of $Y$. Our method is quite elementary but tricky. We find a birational model of $Y$, which is an affine quintic hypersurface of a very simple form (Theorem \ref{BirationalModel}), admitting a conic bundle structure with a rational section over the affine space ${\mathbf C}^2$. 

Our construction is originally motivated by the following closely related question asked by Kenji Ueno (\cite[Page 199]{Ue75}) in 1975 and asked again by Fr\'ed\'eric Campana (\cite{Ca12}) in 2012:

\begin{question}\label{ueno}
Let $E_{\sqrt{-1}} := {\mathbf C}/({\mathbf Z} + {\mathbf Z}\sqrt{-1})$ be the elliptic curve of period $\sqrt{-1}$. 
Let $Z$ be the canonical resolution (\cite[Page 199]{Ue75}) 
of the quotient threefold 
$$E_{\sqrt{-1}} \times E_{\sqrt{-1}} \times E_{\sqrt{-1}} / \langle {\rm diag}\, (\sqrt{-1}, \sqrt{-1}, \sqrt{-1}) \rangle\,\, ,$$
i.e., the blow up at the singular points of type $1/2(1,1,1)$ and 
$1/4(1,1,1)$. 
Is this $Z$ rational or unirational?
\end{question} 

\begin{remark}\label{CurrentProgress}

(1) Campana himself showed that $Z$ is rationally connected (\cite{Ca12}).

(2) Using a method which is similar to, but slightly more involved than, that in this paper, in a joint work with Fabrizio Catanese, we showed that $Z$ is unirational. This is done by finding an explicit affine quintic hypersurface, say $Z'$, birationally equivalent to $Z$ (\cite{COT13}, September 2013). 

(3) Using Brauer's group theory and the explicit equation of $Z'$ in \cite{COT13}, Colliot-Th\'el\`ene finally proved that $Z'$ rational, whence so is $Z$ (\cite{CTh13}, October 2013).

(4) However, since $E_{\sqrt{-1}} \times E_{\sqrt{-1}} \times E_{\sqrt{-1}} / \langle {\rm diag}\, (-1, -1, -1) \rangle\,\, ,$ is already a minimal threefold with 64 singular points of type $1/2(1,1,1)$, we can {\it not} construct a smooth Calabi-Yau threefold from $E_{\sqrt{-1}} \times E_{\sqrt{-1}} \times E_{\sqrt{-1}}$ as in Theorem \ref{main}. 

\end{remark}

The remaining of this paper is organized as follows. In Section 2 we show that $Y$ is rational. In Section 3 we will recall sufficient details on dynamical degrees and relative dynamical degrees for use in Section 4. In Section 4, we first prove our numerical criterion of primitivity of birational automorphisms (Theorem \ref{Restriction}) by using the results in Section 3. Then we construct primitive automorphisms of positive entropy of both $X$ and $Y$ in a uniform manner. 

{\bf Ackowledgement.} We would like to thank Professors Serge Cantat, Fr\'ed\'eric Campana, and De-Qi Zhang for valuable communications relevant to this work. We would like to dedicate this paper to Professor Fr\'ed\'eric Campana on the occasion of his 60-th birthday. In fact, this paper is much motivated by the above mentioned question \ref{ueno} in \cite{Ca12}. 

\section{Rationality of $Y$}

In this section, we shall show that $Y$ in Theorem \ref{main} is rational.

\begin{lemma}\label{Affine}
$E_{\omega}$ is the projective non-singular model, say $\tilde{C}$, of the affine curve 
$$C : y^2 = x^3 -1\,\, .$$ Moreover, the complex multiplication $-\omega$ on $E_{\omega}$ is the extension $\tilde{g}$ of the automorphism of $C$ given by
$$g^* (x, y) \mapsto (\omega x, -y)\,\, .$$
\end{lemma}

\begin{proof} The projective curve $y^2t=x^3-t^3$ is a smooth cubic in $\mathbf{P}^2$ with homogeneous coordinates $[t:x:y]$, hence it defines an elliptic curve $\tilde{C}$ which is the compactification of $C$ in $\mathbf{P}^2$.  The rational $1$-form $dx/y$ on $C$ defines a regular $1$-form $\sigma$ of $\tilde{C}$. Since $g^*dx/y = -\omega dx/y$ and the point $O=[0:0:1]$ is a fixed point, the unique extension $\tilde{g}$ of $g$ defines an automorphism of $\tilde{C}$ such that $\tilde{g}^*\sigma = -\omega \sigma$ with a fixed point $O$. Thus, identifying the origin of $E_{\omega}$ with the point $O$, we obtain $(E_{\omega},-\omega) \simeq (\tilde{C}, \tilde{g})$.
\end{proof}

Let $k \in \{1,2,3\}$ and let ${\mathbf C}_k^2$ be the affine plane with coordinates $(x_k, y_k)$, and 
$$C_k : y_k^2 = x_k^3 -1\,\, ,\,\, g_k : (x_k, y_k) \mapsto (\omega x_k, -y_k)\,\, ,$$
$$V := C_1 \times C_2 \times C_3\,\, ,\,\, g = g_1 \times g_2 \times g_3\,\, .$$
By Lemma \ref{Affine}, our $Y$ is birationally equivalent to the affine threefold $W := V/\langle g \rangle$. 

\begin{lemma}\label{Invariant} The affine coordinate ring ${\mathbf C}[W]$ of $W$ is the subring  
$$R := {\mathbf C}[y_1^2, y_2^2, y_3^2, y_1y_2, y_2y_3, y_3y_1, x_1^ix_2^jx_3^k\, (0 \le i, j, k \le 2\, ,\, i+j+k = 3)]\,\, ,$$ 
of the quotient ring
$${\mathbf C}[y_1, y_2, y_3, x_1, x_2, x_3]/(y_1^2 -x_1^3 +1, y_2^2 -x_2^3 +1, 
y_3^2 -x_3^3 +1)\,\, .$$
\end{lemma}

\begin{proof} Since $V = C_1 \times C_2 \times C_3$, the affine coordinate ring of $V$ is
$${\mathbf C}[y_1, x_1]/(y_1^2 -x_1^3 +1) \otimes_{{\mathbf C}} {\mathbf C}[y_2, x_2]/(y_2^2 -x_2^3 +1) \otimes_{{\mathbf C}} 
{\mathbf C}[y_3, x_3]/(y_3^2 -x_3^3 +1)\,\, .$$
This ring is naturally isomorphic to
$$\tilde{R} := {\mathbf C}[y_1, y_2, y_3, x_1, x_2, x_3]/(y_1^2 -x_1^3 +1, y_2^2 -x_2^3 +1, y_3^2 -x_3^3 +1)\,\, .$$
Note that the subring ${\mathbf C}[y_k]$ of ${\mathbf C}[y_k, x_k]/(y_k^2 -x_k^3 +1)$ is isomorphic to the polynomial ring of one variable and ${\mathbf C}[y_k, x_k]/(y_k^2 -x_k^3 +1)$ is a free 
${\mathbf C}[y_k]$-module with free basis $\langle 1, x_k, x_k^2 \rangle$. 
Therefore the subring ${\mathbf C}[y_1, y_2, y_3]$ of $\tilde{R}$ is isomorphic to the polynomial ring of the three variables and $\tilde{R}$ is also a free ${\mathbf C}[y_1, y_2, y_3]$-module with the free basis 
$$x_1^ix_2^jx_3^k\,\, ,\,\, i, j, k \in \{0, 1,2\}\,\, .$$
In particular, each element of $\tilde{R}$ is uniquely expressed in the following form:
$$f := \sum_{(i, j, k) \in \{0, 1, 2\}^3} a_{ijk}x_1^ix_2^jx_3^k$$
where $a_{ijk} = a_{ijk}(y_1, y_2, y_3)$ are polynomials of $y_1$, $y_2$, $y_3$.
The action of $g$ on $\tilde{R}$ is given by $gy_k= -y_k$ and $gx_k = \omega x_k$ 
and the affine coordinate ring of $V/\langle g \rangle$ is isomorphic to the invariant ring $\tilde{R}^{g}$. Here and hereafter, we denote the action of $g$ 
on $\tilde{R}$, which is $g^*$, simply by $g$. 

Since $\langle g \rangle = \langle g^2, g^3 \rangle$, we have 
$$\tilde{R}^{g} = (\tilde{R}^{g^2})^{g^3}\,\,.$$
Note that $g^2y_m = y_m$ and $g^2x_m = \omega^2 x_m$ ($m=1,2,3$). 
Hence, any polynomial $a_{ijk}(y_1, y_2, y_3)$ are $g^2$-invariants and therefore $f \in R = \tilde{R}^{g^2}$ if and only if $x_1^ix_2^jx_3^k$ are all $g^2$-invariant (for $a_{i,j,k} \not= 0$). That is, $i+j+k$ is divisible by $3$:
$$3 \vert (i + j + k)\,\, .$$
Note then that $i+j+k = 0$, $3$ and $6$ and the term of $i+j+k = 6$ is only 
$(x_1x_2x_3)^2$ by $0 \le i, j, k \le 2$. Hence 
$$\tilde{R}^{g^2} = {\mathbf C}[y_1, y_2, y_3, x_1^ix_2^jx_3^k\, (0 \le i, j, k \le 2\,\, ,\,\, i+j+k = 3)]\,\, .$$
Since $g^3x_m = x_m$ and $g^3y_m = -y_m$ ($m = 1,2, 3$), it follows that any polynomial in $x_1,x_2,x_3$ is $g^3$- invariant. Hence $f$ is $g^3$-invariant if and only if for any monomial $y_1^{i_1}y_2^{i_2}y_3^{i_3}$ appearing in $a_{ijk}$ we have $i_1+i_2+i_3=$ an even number. Therefore 
$$(\tilde{R}^{g^2})^{g^3} = {\mathbf C}[y_my_n\, (1 \le m \le n \le 3)\, , \, x_1^ix_2^jx_3^k\, (0 \le i, j, k \le 2\,\, ,\,\, i+j+k = 3)]\,\, ,$$
as claimed.
\end{proof}

Let $Q(\tilde{R})$ be the quotient field of $\tilde{R}$.

\begin{lemma}\label{QuotientField} The rational function field ${\mathbf C}(Y)$ of $Y$, hence the function field of $W$, 
is isomorphic to the following subfield of $Q(\tilde{R})$: 
$${\mathbf C}(y_1^2, \frac{y_2}{y_1}, \frac{y_3}{y_1}, \frac{x_2}{x_1}, \frac{x_3}{x_1})\,\, .$$ 
\end{lemma}

\begin{proof} It is clear that the field above, say $K$, is a subfield of ${\mathbf C}(W)$ (e.g. $y_2/y_1 = (y_2^2)/(y_1y_2)$ and $x_2/x_1 = (x_2^2x_1)/(x_1^2x_2)$). It suffices to show that each generator of ${\mathbf C}[W]$ in Lemma \ref{Invariant} is in $K$. For $y_my_n$, we have  
$$y_my_n = y_1^2\frac{y_m}{y_1}\frac{y_n}{y_1}\,\, \in K\,\, .$$
For $x_1^ix_2^jx_3^k$ with $i+j+k=3$, we have 
$$x_1^ix_2^jx_3^k = (\frac{x_2}{x_1})^i(\frac{x_3}{x_1})^j(x_1)^3 = (\frac{x_2}{x_1})^i(\frac{x_3}{x_1})^j(y_1^2+1)\,\, \in K\,\, .$$
Hence the result follows.
\end{proof}

\begin{lemma}\label{AbstractField} The rational function field ${\mathbf C}(Y)$ of $Y$ 
is isomorphic to the following subfield of $Q(\tilde{R})$:
$${\mathbf C}(t, s, z, w)\,\, ,$$
with a single equation 
$$(w^3 - 1)(t^2 -1) = (z^3 - 1)(s^2 -1)\,\, .$$
More precisely, the subfield ${\mathbf C}(s, z, w)$ of $Q(\tilde{R})$ is isomorphic to the purely transcendental extension of ${\mathbf C}$ of transcendental degree $3$, and the field ${\mathbf C}(Y)$ is isomorphic to the quotient ring
$${\mathbf C}(s, z, w)[T]/I\,\, ,$$ 
where $I$ is the principal ideal generated by 
$$(w^3 -1)(T^2 - 1)- (z^3 -1)(s^2 - 1)\,\, ,$$
in the polynomial ring ${\mathbf C}(s, z, w)[T]$ over ${\mathbf C}(s, z, w)$.  
\end{lemma}

\begin{proof} Set 
$$(1)\,\,\, u := y_1^2\,\, ,\,\, t := \frac{y_2}{y_1}\,\, ,\,\, s := \frac{y_3}{y_1}\,\, z := \frac{x_2}{x_1}\,\, ,\,\, w := \frac{x_3}{x_1}\,\, .$$
By Lemma \ref{QuotientField}, these $5$ elements are generators of ${\mathbf C}(W) = {\mathbf C}(Y)$.  
Then 
$$(2)\,\,\, y_2= ty_1\,\, ,\,\, y_3 = sy_1\,\, ,\,\, y_2^2 = t^2u\,\, ,\,\, y_3^2 = s^2u\,\, .$$
Hence
$$(3)\,\,\, z^3 =\frac{x_2^3}{x_1^3}=\frac{y_2^2+1}{y_1^2+1}=  \frac{t^2u+1}{u+1}\,\, ,\,\, w^3 = \frac{x_3^3}{x_1^3}=\frac{y_3^2+1}{y_1^2+1}=  \frac{s^2u+1}{u+1}\,\, .$$
Solving each equation in (3) in the variable $u$, we obtain
$$u = \frac{z^3 -1}{t^2 - z^3}\,\, ,\,\, u = \frac{w^3 -1}{s^2 - z^3}$$
that is, 
$$(4)\,\,\, \frac{-1}{u} = \frac{z^3 -t^2}{z^3 - 1} = 1 - \frac{t^2-1}{z^3 -1}$$and
$$(5)\,\,\, \frac{-1}{u} = \frac{w^3 -ts^2}{w^3 - 1} = 1 - \frac{s^2-1}{w^3 -1}\,\, .$$
Hence by (4) and (5), we find that ${\mathbf C}(W) = {\mathbf C}(t, s, z, w)$. Here $z$, $w$, $s$, $t$ satisfy a relation
$$\frac{t^2-1}{z^3 -1} = \frac{s^2-1}{w^3 -1},$$
which is the same as
$$(6)\,\,\, (w^3-1)(t^2-1) = (z^3 -1)(s^2 -1)\,\, .$$
Since $W$ is a $3$-dimensional projective variety, it follows from (6) that 
$s$, $z$, $w$ are a transcendental basis of ${\mathbf C}(W)$ and $t$ is 
algebraic over the field ${\mathbf C}(s, z, w)$. Therefore 
$${\mathbf C}(W) = {\mathbf C}(s, z, w)(t) = {\mathbf C}(s, z, w)[t]\,\, .$$ 
The equation (6), or more precisely, the polynomial
$$(7)\,\,\, (w^3-1)(T^2-1) - (z^3 -1)(s^2 -1)\,\,$$ 
gives an algebraic equation of $t$ over ${\mathbf C}(s, z, w)$. Since the polynomial (7) is irreducible over ${\mathbf C}(s, z, w)$, it follows that
$${\mathbf C}(s, z, w)[t] = {\mathbf C}(s, z, w)[T]/I$$
as claimed. 
\end{proof}

\begin{theorem}\label{BirationalModel} $Y$ is birationally equivalent to the affine hypersurface 
$Q$ defined by
$$(w^3 -1)(t^2 - 1) = (z^3 - 1)(s^2 - 1)$$
in the affine space ${\mathbf A}^4$ 
with affine coordinates $(t, s, z, w)$. 
\end{theorem}

\begin{proof} Lemma \ref{AbstractField} means that the rational function field of $Y$ is isomorphic to the rational function field of the affine hypersurface $Q$. This implies the result.  
\end{proof}

The next proposition completes the proof of Theorem \ref{main} (1):
\begin{proposition}\label{Rational} The hypersurface $Q$ in Theorem \ref{BirationalModel} is rational.
\end{proposition}

\begin{proof} The restriction $\pi_{Q}$ of the natural projection to the last two factors factors
$$\pi : {\mathbf A}^4 \rightarrow {\mathbf A}^2\,\, ,\,\, (t,s,z,w) \mapsto (z, w)$$
gives a conic bundle structure of $Q$ over ${\mathbf A}^2$ with an obvious section $(s, t) = (1,1)$. Let $\eta$ be the generic point of ${\mathbf A}^2$ in the sense of scheme. Then the fiber $C_{\eta} := \pi_{Q}^{-1}(\eta)$ 
is the conic defined by 
$$(w^3 - 1)(t^2 - 1) = (z^3 - 1)(s^2 - 1)$$
in the affine plane $\pi^{-1}(\eta) = {\mathbf A}_{\eta}^{2}$ with affine coordinates $(s, t)$ over ${\mathbf C}({\mathbf A}^2) = {\mathbf C}(z, w)$. 
Moreover, the point $(s, t) = (1,1)$ 
is clearly a rational point of $C_{\eta}$ also 
over ${\mathbf C}({\mathbf A}^2)$. 
Hence $C_{\eta}$ is birational to the affine line ${\mathbf A}_{\eta}^{1}$ over ${\mathbf C}({\mathbf A}^2)$. For this, one may just consider the projection from the above rational point to ${\mathbf A}_{\eta}^{1}$ also over ${\mathbf C}({\mathbf A}^2)$. Note that the generic point $\tilde{\eta}$ of $Q$ is lying over 
$\eta$. Denote the residue field of $\tilde{\eta}$ (resp. of $\eta$) by $K(\tilde{\eta})$ (resp. $K(\eta)$). Then $K(\eta) = {\mathbf C}({\mathbf A}^2) = {\mathbf C}(z, w)$, $K(\eta) \subset K(\tilde{\eta})$ and
$${\mathbf C}(Q) = K(\tilde{\eta}) = K(\eta)(C_{\eta}) \simeq K(\eta)({\mathbf A}_{\eta}^{1}) = {\mathbf C}(z, w)(\gamma) = {\mathbf C}( z,w, \gamma)\,\, ,$$
where $\gamma$ is the affine coordinate function of ${\mathbf A}_{\eta}^{1}$.  
Since ${\rm dim}\, Q = 3$, the rational functions $z$, $w$ and $\gamma$ are algebraically independent over ${\mathbf C}$. Hence ${\mathbf C}(Q)$ is purely transcendental over ${\mathbf C}$. This implies the result.
\end{proof}

\begin{remark}\label{anotherproof} Finding a "good" explicit birational affine hypersurface as above is also crucial in both \cite{COT13} and \cite{CTh13}. Only for the rationality of our $W$, we can prove more explicitly as follows. Put 
$$\gamma :=(s-1)/(t-1)$$ 
in ${\mathbf C}(W)$. Then $s = \gamma (t-1)+1$ 
and from $(w^3-1)(t^2 - 1)=(z^3-1)(s^2 -1)$, we obtain a relation
\begin{eqnarray*}
(w^3-1)(t+1)=(z^3-1)\gamma (\gamma (t-1)+2).
\end{eqnarray*}
Since this relation is linear in $t$, we can solve $t$ in terms of $w,z,\gamma$. Therefore $\mathbf{C}(W) = \mathbf{C}(t, s, z, w)$ is generated by $w,z,\gamma$ over $\mathbf{C}$. Since $W$ has dimension $3$, this implies that $W$ is rational. 
\end{remark}

\begin{remark}\label{FieldDefinition} $\tilde{C}$ in Lemma \ref{Affine} is the elliptic curve defined over any field $K$ of characteristic $\not= 2, 3$ 
and the automorphism $\tilde{g}$ is defined over any field $K$ containing the primitive third root of unity $\omega$. The argument in this section shows that 
$V/\langle g \rangle$ is birationally equivalent to $Q$ and it is rational, also over any field $K$ containing $\omega$ and of characteristic $\not= 2$, $3$.
\end{remark}

\begin{remark}\label{Higher} Consider the quotient variety of dimension $n \ge 2$:
$$V_n := E_{\omega}^n/ \langle -\omega I_n \rangle\,\, .$$
Then $V_n$ has isolated singular points of type $1/2(1, 1, \cdots , 1)$, $1/3(1, 1, \cdots , 1)$ and one isolated singular point $O$ of type $1/6(1, 1, \cdots , 1)$. From this, it is easy to see that ${\mathcal O}_{V_n} (6K_{V_n}) \simeq {\mathcal O}_{V_n}$, $h^1({\mathcal O}_{V_n}) = 0$ and 
the singular point $O$ is Kawamata log terminal but not canonical when $2 \le n \le 5$, canonical but not terminal when $n = 6$ and terminal (but not smooth) when $n \ge 7$. Also all other singular points are terminal when $n \ge 4$. (See \cite{KMM87}, \cite{KM98} for terminologies and basic notions of minimal model theory.) So, when $n \ge 7$ (resp. when $n= 6$), $V_n$ (resp. the blow up of $V_6$ at $O$) are minimal but singular Calabi-Yau varieties, in the sense of minimal model theory, and therefore, they are not even uniruled. On the other hand, the Kodaira dimension of the resolution of $V_n$ ($2 \le n \le 5$ ) is $-\infty$. 
So, $V_2$ is rational by Castelnouvo's criterion (see eg. \cite[Page 252]{BHPV04}). Our result shows that $V_3$ is rational. It would be interesting to see if $V_4$, $V_5$ are rational, unirational or not. 
\end{remark}

\section{Topological entropy, Dynamical degrees and relative dynamical degrees}

In this section we recall some known facts about topological entropy, dynamical degrees and relative dynamical degrees which will be used later. Our main references are \cite{dinh-sibony10}, \cite{dinh-sibony3}, \cite{dinh-nguyen} and \cite{dinh-nguyen-truong}.

\subsection{Topological entropy}

 Let $f:X\rightarrow  X$ be a surjective holomorphic map. Then $f$ is in particular continuous, and by the classical ergodic theory we can define the topological entropy of $f$. Its topological entropy measures how the map separates the orbits of distinct points, and hence is an indication of the complexity of $f$. The larger topological entropy, the more complexity. Its definition is given by: 
\begin{eqnarray*}
h_{top}(f)=\sup _{\epsilon >0} (\limsup _{n\rightarrow\infty}\log \max \{\# F: F\mbox{ is an } (n,\epsilon ) ~set\}).
\end{eqnarray*}  
Here a set $F$ is called an $(n,\epsilon )$ set if any two distinct points $x$ and $y$ in $F$ are $(n,\epsilon )$ separated, that is the distance between the two $n$-orbits $(x,f(x),f^2(x),\ldots ,f^n(x))$ and $(y,f(y),f^2(y),\ldots ,f^n(y))$ is at least $\epsilon$. Since $X$ is compact, any $(n,\epsilon )$ set is finite. 

The topological entropy $h_{top}(f)$ of a meromorphic map can be defined similarly (by the works of Gromov, Friedland, Guedj, Dinh and Sibony) but more complicated than that of a holomorphic map, see e.g. Guedj \cite{guedj2}. 

\subsection{Dynamical degrees}   

A meromorphic map $f:X\cdots \to Y$ between two complex manifolds is a holomorphic map $f|_U:U\rightarrow Y$ from a Zariski open dense set $U$ of $X$ into $Y$ so that the closure of the graph of $f|_U$ is an analytic subvariety of $X\times Y$.  We say that $f$ is dominant if $f|_U(U)$ is dense in $Y$. An example of rational maps is a rational function $f$ on $\mathbb{C}$ of the form $f(z)=P(z)/Q(z)$ where $P$ and $Q$ are polynomials in the variable $z\in \mathbb{C}$ which are relatively prime. While $f$ is not a continuous selfmap on $\mathbb{C}$, it becomes a holomorphic selfmap when extended to the complex projective line $\mathbb{P}^1$. $f$  has two dynamical degrees: $\lambda _0(f)=1$ and $\lambda _1(f)=\max\{{deg(P)},{deg(Q)}\}$ the topological degree of $f$. The importance of these dynamical degrees was shown already since the works of Gaston Julia, Pierre  Fatou and Lucjan B\"ottcher more than 100 years ago (see Milnor's book \cite{milnor}). For example, a classical result says that if $\lambda _1(f)>1$ then the periodic points of $f$ are equidistributed with respect to its equilibrium measure (see Brolin \cite{brolin}, Freie, Lopes and Mane \cite{freie-lopes-mane}, Lyubich \cite{lyubich} and Tortrat \cite{tortrat}).

Since the fundamental results of Gromov and Yomdin, dynamical degrees have proved important in dynamics of holomorphic and meromorphic selfmaps in higher dimensions as well. In many results and conjectures in complex dynamics in higher dimensions, dynamical degrees play a central role.

In what follows, let $X$ be a compact K\"ahler manifold of dimension $k$, and $f:X\cdots \to X$ a dominant meromorphic map. (In this section, $X$ and $Y$ have nothing to do with $X$ and $Y$ in our main theorem).

The pullback maps $f^*:H^{p,p}(X)\rightarrow H^{p,p}(X)$ are well-defined. The idea is that if $\theta$ is a smooth form then we can define $f^*(\theta )$ as a current, that is the extension by zero of the $(p, p)$-form $(f\vert U)^*(\theta)$. Here $U$ is any Zariski dense open set where $f$ is holomorphic and $(f\vert U)^*(\theta)$ is the usual pullback by the holomorphic $f \vert U$. Note however that in general  we do not have the compatibility of pullback maps: $(f^n)^*$ may be different from $(f^*)^n$ on $H^{p,p}(X)$. Russakovskii and Shiffman \cite{russakovskii-shiffman} (the case $X=\mathbb{P}^k$) and Dinh and Sibony \cite{dinh-sibony1, dinh-sibony10} (the case of compact K\"ahler manifolds) defined the $p$-th dynamical degree $\lambda _p(f)$ as follows 
\begin{equation}
\lambda _p(f)=\lim _{n\rightarrow\infty}r_p(f^n)^{1/n},
\label{EquationDynamicalDegreeDefinition}\end{equation}
where $r_p(f^n)$ is the spectral radius of $(f^n)^*:H^{p,p}(X)\rightarrow H^{p,p}(X)$. The existence of the limit in (\ref{EquationDynamicalDegreeDefinition}) is a non-trivial fact. The main idea is to show that there is a constant $C>0$, being independent of the maps $f,g:X\rightarrow X$, such that $r_p(f\circ g )\leq Cr_p(f)r_p(g)$. Proving this needs good regularization of a positive closed $(p,p)$ current, and such a regularization was given for a compact K\"ahler manifold in \cite{dinh-sibony1, dinh-sibony10}. 

Here are some simple properties of dynamical degrees (\cite{dinh-sibony10}): 
$$\lambda _0(f)=1\,\, ,\,\, \lambda _k(f)=\, {\rm deg}\, f\,\, ,$$ 
where ${\rm deg}\, f$ is the topological degree of $f$, and the log-concavity 
$$\lambda _{p-1}(f)\lambda _{p+1}(f)\leq \lambda _p(f)^2\,\, .$$ 
In particular, $\lambda _p (f) \ge 1$ for $0 \le p \le k = \dim\, X$.
Given that the limit in (\ref{EquationDynamicalDegreeDefinition}) exists, it is easy to show that dynamical degrees are bimeromorphic invariants (\cite{dinh-sibony10}). 

For a surjective holomorphic map $f$, $\lambda _p(f)$ is simply $r_p(f)$. In this case, we have 
$$h_{top}(f)=\max _{0\leq p\leq k}\log \lambda _p(f)\,\, ,$$ 
where $h_{top}(f)$ is the topological entropy of $f$. Gromov \cite{gromov} proved the inequality $h_{top}(f)\leq \max _{0\leq p\leq k}\log \lambda _p(f)$, and Yomdin \cite{yomdin} proved the reverse inequality:
$$h_{top}(f) \ge r(f^* \vert H^*(X, {\mathbf R}))\,\, .$$
Here $r(f^* \vert H^*(X, {\mathbf R}))$ is the spectral radius of $f^*$ on the total cohomology ring $H^*(X, {\mathbf R})$. The proof of Yomdin's inequality is quite involved and it holds more generally for $C^{\infty}$ maps.  For a meromorphic map, Gromov's inequality still holds (\cite{dinh-sibony1, dinh-sibony10}), but Yomdin's inequality does not hold in general (\cite{guedj2}).

\begin{remark}\label{SingularSetting} 

(1) We can define dynamical degrees in a more general setting. Let $X$ be a compact complex variety, not necessarily smooth, which is bimeromorphically equivalent to a compact K\"ahler manifold $\widetilde{X}$ via a bimeromorphic map $\pi :\widetilde{X} {\cdots } \to X$ (such a variety $X$ is usually said to be of Fujiki's class {\it C}). If $f:X\cdots \to X$ is a dominant meromorphic map, we can define dynamical degrees of $f$ as follows. Let $$\widetilde{f}:=\pi ^{-1}\circ f\circ \pi:~\widetilde{X}\cdots \to \widetilde{X}$$ 
be the induced map. Then we define 
$$\lambda _p(f):=\lambda _p(\widetilde{f})\,\, .$$
This is well-defined by the bimeromorphic invariance of dynamical degrees for compact K\"ahler {\it manifolds}. These facts will be used in our proof of Theorem \ref{main} (2). 

(2) Using Chow's moving lemma, we can also define dynamical degrees algebraically for a rational selfmap over an arbitrary algebraic closed field. The case of characteristic zero is treated in \cite{truong2}. The general case and applications to dynamics over a non-Archimedean field will be given in an ongoing joint project of C. Favre and the second author. 
\end{remark}    

\subsection{Relative dynamical degrees}

If a meromorphic map preserves a meromorphic fibration, then there are some relations which must be satisfied by its dynamical degrees. Specifically, let 
$$f:X\cdots \to X\,\, ,\,\, g:Y\cdots \to Y\,\, ,\,\, 
\pi :X \cdots \to Y$$ 
be dominant meromorphic maps such that $\pi \circ f=g\circ \pi$, 
where $X$ and $Y$ are compact K\"ahler manifolds of dimensions $k$ and $l$ 
with $k\geq l$. Let $\omega _X$ be a K\"ahler form on $X$ and $\omega _Y$ a K\"ahler form on $Y$. 

In the case where $\pi$ is {\it holomorphic}, the relative dynamical degrees $\lambda _p(f|\pi )$ (here $0\leq p\leq k-l$) is defined by:
\begin{eqnarray*}
\lambda _p(f|\pi )=\lim _{n\rightarrow\infty}(\int _X(f^n)^*(\omega _X^p)\wedge \pi ^*(\omega _Y^{l})\wedge \omega _X^{k-l-p})^{1/n}.
\end{eqnarray*} 
Here  $(f^n)^*(\omega _X^p)$ is defined as a current and $\pi ^*(\omega _Y^{l})\wedge \omega _X^{k-p-l}$ is defined as the usual smooth form, and therefore the integration above makes sense.  
This definition is due to Dinh and Nguyen \cite{dinh-nguyen}. They also proved that $\lambda _p(f|\pi )$ satisfy the log-concavity property and that they are bimeromorphic invariants in the sense that 
$$\lambda _p(f|\pi ) = \lambda _p(\widetilde{f}|\widetilde{\pi} )$$ for 
two surjective holomorphic maps between compact k\"ahler manifolds 
$$\pi : X \rightarrow Y\,\, ,\,\, \widetilde{\pi} : \widetilde{X} \rightarrow \widetilde{Y}$$ 
being bimeromorphic in an obvious sense. Here $\widetilde{f}$, $\widetilde{g}$ are the meromorphic selfmaps of $\widetilde{X}$ and $\widetilde{Y}$ induced by 
$f$ and $g$. Note also that $\lambda _0(f|\pi ) =1$ 
and $\lambda _{k-l}(f|\pi )$ is the topological degree of $f\vert X_t : X_t \cdots\to X_{g(t)}$ for generic fibers $X_t$ ($t \in Y$). The log-concavity property then implies that $\lambda _p(f|\pi ) \ge 1$ for any {\it meaningful} $p$, i.e., for all integers $p$ such that $0 \le p \le k-l$. 

Using the bimeromorphic invariance, we can also define $\lambda _p(f|\pi )$ when $\pi$ is not necessarily holomorphic: First take any resolution 
$$\mu : \widetilde{X} \rightarrow X$$
of the indeterminacy of $\pi$ by a compact K\"ahler manifold $\widetilde{X}$. 
Then the maps 
$$\widetilde{f} := \mu^{-1} \circ f \circ \mu : \widetilde{X} \cdots \to \widetilde{X}\,\, ,\,\, g:Y\cdots \to Y\,\, ,\,\, \widetilde{\pi} := \mu \circ \pi : \widetilde{X} \rightarrow Y$$ 
are dominant meromorphic maps such that $\tilde{\pi} \circ \tilde{f} = g \circ \tilde{\pi}$ and $\tilde{\pi}$ is {\it holomorphic}. We define the relative dynamical degree by 
$$\lambda _p(f|\pi ) := \lambda _p(\widetilde{f} |\widetilde{\pi})\,\, .$$ 
Well-definedness follows from the bimeromorphic invariance. By definition, 
$\lambda _p(f|\pi )$ also satisfy the log-concavity property, and that $\lambda _0(f|\pi ) =1$, $\lambda _{k-l}(f|\pi )$ is the topological degree of $f\vert X_t$ as before and $\lambda _p(f|\pi ) \ge 1$ for any meaningful $p$. 

The following result was proved in \cite{dinh-nguyen-truong}. Some special cases were proved in \cite{dinh-nguyen} and  \cite{nakayama-zhang}.
\begin{theorem}\label{Dynamical} Let 
$$f:X\cdots \to X\,\, ,\,\, g:Y\cdots \to Y\,\, ,\,\, 
\pi :X \cdots \to Y$$ 
be dominant meromorphic maps such that $\pi \circ f=g\circ \pi$, 
where $X$ and $Y$ are compact K\"ahler manifolds of dimensions $k$ and $l$ 
with $k\geq l$. Then for all $0\leq p\leq k$:
\begin{equation}
\lambda _p(f)= \max_{\max\{0,p-k+l\}\leq j\leq \min\{p,l\}}\lambda _j(g)\lambda _{p-j}(f|\pi ). \label{EquationKahlerCase}\end{equation}
\end{theorem}
We can see (\ref{EquationKahlerCase}) easily in the special case $X=Y\times Z$, $f=(g,h)$ a product map, and $\pi :X\rightarrow Y$ is the projection to $Y$ (by using the Kunneth's formula for the cohomology groups of $X$). In the general case, the main tool is an analogous Kunneth's formula, which now is an inequality rather than an equality. More precisely, if $T$ is a positive closed $(p,p)$ current which is smooth on a Zariski open dense set and has no mass on proper subvarieties, then in cohomology:
\begin{equation}
\{T\}\leq A\sum _{\max \{0,p-k+l\}\leq j\leq \min \{l,p\}}\alpha _j(T)\{\pi ^*(\omega _Y^j)\}\smile \{\omega _X^{p-j}\}, \label{eqn_alpha}
\end{equation}
 where 
\begin{equation*} 
\alpha_j(T):=\big\langle T,\pi^*(\omega_Y^{l-j})\wedge
\omega_X^{k-l-p+j}\big\rangle.
\end{equation*}
In fact, we need a stronger version of (\ref{eqn_alpha}) which deals with positive closed currents instead of only cohomology classes.  

If $X=Y\times \mathbb{P}^{k-l}$, where $Y$ is projective, Dinh and Nguyen \cite{dinh-nguyen} proved (\ref{eqn_alpha}) and its positive closed current version using Kunneth's formula and the fact that $\mathbb{P}^{k-l}$ has a lot of automorphisms. They then prove (\ref{EquationKahlerCase}) for dynamical degrees in case $X$ and $Y$ are projective. An essential point used in their proof is that any dominant rational map $\pi :X\cdots \to Y$ is, up to a finite covering, the canonical projection $Y\times \mathbb{P}^{k-l} \rightarrow Y$. 

For the case of a dominant meromorphic map of  compact K\"ahler manifolds $\pi :X\cdots \to Y$, it is not known whether $\pi$ is, up to a finite covering, a canonical projection in the sense above. Instead, \cite{dinh-nguyen-truong} proceeded as follows. Let $T$ be a positive closed current on $X$, and let $\Delta _X$ be the diagonal. Then we have $T=(\pi _2)_*(\pi _1^*(T)\wedge [\Delta _X])$ (here $\pi _1$ and $\pi _2$ are the natural projections), and hence it is enough to prove a similar formula for $T=[\Delta _X]$ and $\pi$ is replaced by the product $\pi \times \pi :X\times X\cdots \to Y\times Y$. To this end, we observe that $\Delta _X$ is a subvariety of $(\pi \times \pi )^{-1}(\Delta _Y)$. Then we extend the semi-regularization of Dinh and Sibony to the form: if $V$ is a submanifold of $W$ and $T$ is a positive closed current on $V$, then $T$ can be semi-regularized by currents of the form $\iota ^*( \theta _n)$, where $\iota :V\subset W$ is the inclusion of $V$ in $W$ and  $\theta _n$ is  a positive closed smooth form on $W$. Dinh and Sibony's regularization corresponds to the case $V=W$.    

\begin{remark}\label{SingularRelativeSetting} 

(1) Because of the bimeromorphic invariance of relative dynamical degrees, we can define relative dynamical degrees and apply Theorem \ref{Dynamical} in a more general setting. Let $X$, $Y$ be compact complex varieties, not necessarily smooth, which are bimeromorphically equivalent to compact K\"ahler manifolds $\widetilde{X}$, $\widetilde{Y}$. Let  
$$\widetilde{f}:\widetilde{X}\cdots\to \widetilde{X}\,\, ,\,\, \widetilde{g}:\widetilde{Y}\cdots\to \widetilde{Y}\,\, ,\,\, \widetilde{\pi}:\widetilde{X} 
\cdots \to \widetilde{Y}$$ 
be the induced maps. We thus define 
$$\lambda _p(f|\pi ):=\lambda _p(\widetilde{f}|\widetilde{\pi})\,\, ,$$
which is well-defined.

(2) In case $\pi :X\cdots\to  Y$ is a bimeromorphic map, or more generally a generically finite meromorphic map, then Theorem \ref{Dynamical} implies that $\lambda _p(f)=\lambda _p(g)$ for every $p$ (note that this case was proved earlier in \cite{dinh-nguyen}). This is because in this case there is only one relative dynamical degree, that is, $\lambda _0(f|\pi )=1$. In particular, in the case where $\pi :X\rightarrow Y=X/G$ is the quotient by a finite group, we have $\lambda _p(f)=\lambda _p(g)$ for every $p$. 

(3) By the previous remark (2), if $\pi :X\rightarrow Y$ is a generically 
finite holomorphic map, $X$, $Y$ are compact K\"ahler manifolds and $f$, $g$ are surjective holomorphic maps with $\pi \circ f = g \circ \pi$, then $h_{top}(f)=h_{top}(g)$. In particular, we will apply for the case $\pi :X\rightarrow Y=X/G$ is the quotient by a finite group and $f$, $g$ are automorphisms, that is, the case where the automorphism $f$ normalizes $G$ and $g$ is the descent of $f$. A priori, $Y$ may not be smooth, so Gromov-Yomdin's theorem may not apply. But if there is a compact K\"ahler manifold $\widetilde{Y}$ birationally equivalent to $Y$ such that the induced map $\widetilde{g}$ is biholomorphic, then we have $h_{top}(\widetilde{g})=h_{top}(f)$. Such a $\widetilde{Y}$ exists, by the equivariant resolution of singularities due to Hironaka. 
\end{remark}

\section{Primitive automorphisms of $X$ and $Y$}

In this section, we prove Theorem \ref{main} (2). We obtain the following strong restriction of the dynamical degrees of imprimitive  rational selfmaps of projective manifolds. 

\begin{theorem}\label{Restriction}
(1) Let $M$ be a smooth projective manifold and $f$ a dominant meromorphic selfmap of $M$. If $\lambda_1(f)>\lambda_2(f)$ then $f$ is primitive.

(2) Let $M$ be a compact K\"ahler manifold of dimension $3$ and $f:M\cdots\to M$ a bimeromorphic map. Assume that $f$ is imprimitive. Then $\lambda_1(f) = \lambda_2(f)$. In other words, if $\lambda_1(f) \not= \lambda_2(f)$, then $f$ has to be primitive. 
\end{theorem}

\begin{remark}\label{known} This kind of criterion should be known to experts in complex dynamics. However, an explicit  statement of the criterion with a complete proof (which relies on recent results in complex dynamics, see Section 3) seems not be available in the literature. This theorem is crucial in our main result Theorem \ref{main} (2). Part (1) of Theorem \ref{Restriction} is in particular useful when considering maps in dimensions at least $4$, please see Remark \ref{final} for more details. Therefore, we present here such a statement and proof together with a brief, but hopefully comprehensible, account of needed results (see Section 3).   
\end{remark} 

\begin{proof} 

We first prove (1). Assume that $f$ is imprimitive. Then there are compact K\"ahler manifold $B$, a dominant meromorphic maps $\pi : M \cdots \to B$, $g : B \cdots \to B$ such that $\pi \circ f=g\circ \pi$ and $0 < \dim\, B < \dim\, M$. 
Note that  
$$\lambda_0(f) = \lambda_0 (g) = \lambda_0(f|\pi ) = 1\,\, ,$$ 
by definition. Then by Theorem \ref{Dynamical}, we have 
$$\lambda_1(f) = \max \{\lambda_1(g),  \lambda_1(f|\pi ) \}\,\, ,\,\, 
\lambda_2(f)=\max \{\lambda_2(f|\pi ), \lambda_1(g)\lambda_1(f|\pi ), \lambda_2(g) \}\,\, .$$ 
{\it Here the maximum is taken among the meaningful ones.} 
By $0 < \dim\, B < \dim\, M$, both terms $\lambda _1(g)$ and $\lambda _1(f|\pi )$ are meaningful, and hence they are greater than or equal to $1$ 
(log-concavity). Therefore 
$$\lambda_2(f) \geq \lambda_1(g)\lambda_1(f|\pi ) \geq \max \{\lambda_1(g), \lambda_1(f|\pi )\}=\lambda_1(f)\,\, .$$ Hence if $\lambda _1(f)>\lambda _2(f)$ then $f$ is primitive. 
 
Now we prove (2). Assume that $f$ is imprimitive. Then there are a compact K\"ahler manifold $B$, dominant meromorphic maps $\pi : M \cdots \to B$, $g : B \cdots \to B$ such that $\pi \circ f=g\circ \pi$ and $0 < \dim\, B < 3 =\dim\, M$. Then as in (1), we have 
$$\lambda_1(f) = \max \{\lambda_1(g),\lambda_1(f|\pi ) \}\,\, ,\,\, \lambda_2(f)=\max \{\lambda_2(f|\pi ),\lambda_1(g)\lambda_1(f|\pi ),\lambda_2(g)\}\,\, .$$
{\it Here again the maximum is taken among the meaningful terms.} 
Moreover, since $f$, and hence $g$, are bimeromorphic, it follows that 
$$\lambda_3(f) = \lambda_{\dim\, B} (g) = 1\,\, .$$

First consider the case where $\dim\, B = 1$. Then, by Theorem \ref{Dynamical}, we also have 
$$1 = \lambda_3(f) = \lambda_1(g)\lambda_2(f|\pi )\,\, .$$
Hence 
$$\lambda_1(g) = \lambda_2(f|\pi ) = 1\,\, .$$
Therefore
$$\lambda_1(f) = \max \{\lambda_1(g),\lambda_1(f|\pi ) \} = \max \{\lambda_1(f|\pi), 1\}\,\, ,$$
$$\lambda_2(f)=\max \{\lambda_2(f|\pi ),\lambda_1(g)\lambda_1(f|\pi ) \} = \max \{\lambda_1(f|\pi), 1\}\,\, .$$
Hence $\lambda_1(f) = \lambda_2(f)$. 

Next consider the case where $\dim\, B = 2$. Then, by Theorem \ref{Dynamical}, we also have 
$$1 = \lambda_3(f) = \lambda_2(g)\lambda_1(f|\pi )\,\, .$$
Hence 
$$\lambda_2(g) = \lambda_1(f|\pi ) = 1\,\, .$$
Therefore
$$\lambda_1(f) = \max \{\lambda_1(g),\lambda_1(f|\pi ) \} = \max \{\lambda_1(g), 1\}\,\, ,$$
$$\lambda_2(f)=\max \{\lambda_1(g)\lambda_1(f|\pi ), \lambda_2(g)\} = \max \{\lambda_1(g), 1\}\,\, .$$
Hence $\lambda_1(f) = \lambda_2(f)$ as well. 
\end{proof}

Now we are ready to prove our main theorem (2). 

Set $M := E_{\omega} \times E_{\omega} \times E_{\omega}$. Then we have a natural embedding of groups 
$${\rm GL}\,(3, {\mathbf Z}[\omega]) \subset {\rm Aut}\, (M)\,\, .$$ 
Moreover, since $\pm {\rm diag}\, (\omega, \omega, \omega)$ is in the center of ${\rm GL}(3, {\mathbf Z}[\omega])$, it follows that any $f \in {\rm GL}(3, {\mathbf Z}[\omega])$ naturally descends to the {\it biregular} automorphisms of $X$ and $Y$. The fact that they are biregular follows from the costruction 
and the universal property of blow-up. We denote them by $f_X$ and $f_Y$ respectively. 

\begin{lemma}\label{Tori} Let $a$ be any positive integer. Consider the automorphism $f := f_a$ of $M$ given by the matrix
$$P = P_a = \left(\begin{array}{rrr}
0 & 1 & 0\\
0 & 0  & 1\\
-1 & 3a^2 & 0\\
\end{array} \right)\,\, .$$
Then $d_2(f) > d_1(f) > 1$. Moreover $d_1(f)$ is not a Salem number.
\end{lemma}
Note here that a {\it Salem number} is a real algebraic integer $a > 1$ with Galois conjugates $a$, $1/a$ such that all other Galois conjugates are 
on the unit circle $S^1 \subset {\mathbf C}$. 
 
\begin{proof} Note that $P \in {\rm SL}(3, {\mathbf Z})$ so that $P$ gives an automorphism $f$ of $M$. Thus $\lambda_p(f)$ is the spectral radius 
of $f^* \vert H^{p, p}(M)$. The characteristic polynomial $\Phi(t)$ of the matrix $P$ is
$$\Phi(x) = x^3 - 3a^2x +1\,\, .$$
Observe $\Phi'(x) = 3(x-a)(x+a)$, and
\begin{eqnarray*}
\Phi (-a)=2a^2+1>0,~\Phi (0)=1>0,~\Phi (1)=2-3a^2<0,~\Phi (a)=-2a^2+1<0\,\, .
\end{eqnarray*}
Hence $\Phi(x)$ has three real roots $\alpha$ $\beta$, $\gamma$ such that 
$$\alpha <-a < 0< \beta <1 <a< \gamma\,\, .$$
We have also 
$$\alpha + \beta + \gamma = 0$$
by the shape of $\Phi(x)$. Hence  
$$0 < \vert \beta \vert < 1 < \vert \gamma \vert < \vert \alpha \vert\,\, .$$
Observe that the characteristic polynomial of $f^* \vert H^1(M, {\mathbf Z})$ is $\Phi(x)^2$. Hence $\alpha$, $\beta$, $\gamma$ are the eigenvalues of $f^* \vert H^1(M, {\mathbf Z})$ each of which is of multiplicity $2$. 
Since 
$$H^{2k}(M, {\mathbf Z}) = \wedge^{2k}H^{1}(M, {\mathbf Z})\,\, ,$$ 
it now follows that $\lambda_1(f) = \alpha^2$ and $\lambda_2(f) = \alpha^2\gamma^2$, therefore $\lambda_2(f) > \lambda_1(f) >1$. 

Since $f$ is an automorphism in dimension $3$ and $\lambda_1(f)\not= \lambda_2(f)$, by the results in \cite{oguiso-truong} we have that $\lambda_1(f)$ is not a Salem number. We can also see this directly as follows. Note that $\Phi(x)$ is irreducible over ${\mathbf Q}$, or equivalently, over ${\mathbf Z}$. Indeed, otherwise, $\Phi(x)$ would have one of $\pm 1$ as its root, a contradiction. Hence the three roots $\alpha$, $\beta$ and $\gamma$ are Galois conjugate over ${\mathbf Q}$. Thus, so are $\alpha^2$, $\beta^2$, $\gamma^2$. Since $\alpha^2 > \gamma^2 > 1$, it follows that $\lambda_1(f) = \alpha^2$ is not 
a Salem number. 
\end{proof}

The next proposition completes the proof of Theorem \ref{main} (2):

\begin{proposition}\label{Fibration} The induced automorphisms $f_X$ and $f_Y$ are primitive and of positive entropy for $f := f_a$ in Lemma \ref{Tori}. Moreover, their first dynamical degrees are not Salem numbers. 
\end{proposition}

\begin{proof} By Remark \ref{SingularRelativeSetting}, we have 
$$\lambda_p(f_X) = \lambda_p(f) = d_p(f_Y)\,\, .$$ 
By Lemma \ref{Tori}, we have $\lambda_2(f) > \lambda_1(f)$. Hence 
$$\lambda_2(f_X) > \lambda_1(f_X) > 1\,\, ,\,\, \lambda_2(f_Y) > \lambda_1(f_Y) >1\,\, .$$ 
Thus $f_X$ and $f_Y$ are primitive by Theorem \ref{Restriction}. they are of positive entropy by $\lambda_1(f_X) >1$ and $\lambda_1(f_Y) >1$. 
By Lemma \ref{Tori}, $\lambda_1(f_X) = \lambda_1(f_Y) = \lambda_1(f)$ is not a Salem number 
as well. 
\end{proof}

\begin{remark}\label{Fin} For the same reason, $f$ in Lemma \ref{Tori} is also a primitive automorphism of the $3$-torus $M$ of positive entropy.
\end{remark}

\begin{remark}\label{final} 

(1) We give only one explicit example of smooth rational threefold and Calabi-Yau threefold in the strict sense, admitting primitive automorphisms of positive entropy. However, we expect that there should be many such smooth rational threefolds and  Calabi-Yau threefolds. It will be interesting to find more examples. 

(2) It will be also very interesting to find rational manifolds and Calabi-Yau manifolds (in both weak and strict senses), {\it of dimension} $\ge 4$, admitting primitive automorphisms of positive entropy. In finding such examples, our theorem \ref{Restriction} (1) might be useful. 
\end{remark}

\end{document}